\documentclass[10pt]{amsart}
\usepackage[margin=1.5in]{geometry}
\usepackage{amssymb,latexsym, mathtools,tikz}
\usepackage{enumerate}
\usepackage{graphicx}
\usepackage{float}
\usepackage{placeins}
\usepackage{mdframed}
\usepackage{amssymb}
\usepackage{esint}
\usepackage{cool}
\usepackage[all,cmtip]{xy}
\usepackage{mathtools}
\usepackage{amstext} 
\usepackage{array}   
\usepackage[shortlabels]{enumitem}
\usepackage{ytableau}
\usepackage{blkarray}
\usepackage{tikz}
\usetikzlibrary{arrows,decorations.markings, cd}
\tikzstyle arrowstyle=[scale=1]
\tikzstyle directed=[postaction={decorate,
decoration={markings,mark=at position .65 with {\arrow[arrowstyle]{stealth}}}}]
\usepackage{mathdots}
\usepackage{quiver}

\newcolumntype{L}{>{$}l<{$}} 
\newcolumntype{C}{>{$}c<{$}}
\newtheorem{theorem}{Theorem}[section]
\newtheorem{lemma}[theorem]{Lemma}
\newtheorem{cor}[theorem]{Corollary}
\newtheorem{prop}[theorem]{Proposition}

\theoremstyle{definition}
\newtheorem{definition}[theorem]{Definition}
\newtheorem{example}[theorem]{Example}

\newtheorem{obs}[theorem]{Observation}
\newtheorem{notation}[theorem]{Notation}

\newtheorem{construction}[theorem]{Construction}

\theoremstyle{remark}
\newtheorem{remark}[theorem]{Remark}

\newtheorem{the context}[theorem]{The Context}
\newtheorem{question}[theorem]{Question}
\numberwithin{equation}{theorem}
\numberwithin{equation}{section}

\newcommand{\rank}{\operatorname{rank}}

\newcommand{\tor}{\operatorname{Tor}}
\newcommand{\im}{\operatorname{Im}}

\newcommand{\cone}{\operatorname{Cone}}
\newcommand{\Ker}{\operatorname{Ker}}

\newcommand{\ideal}[1]{\mathfrak{#1}}
\newcommand{\m}{\ideal{m}}

\newcommand{\bbz}{\mathbb{Z}}
\newcommand{\bbn}{\mathbb{N}}

\newcommand{\x}{\mathbf{x}}

\renewcommand{\geq}{\geqslant}
\renewcommand{\leq}{\leqslant}
\renewcommand{\ker}{\Ker}
\renewcommand{\hom}{\Hom}

\newcommand{\Hom}{\operatorname{Hom}}

\newcommand{\maps}[5]{\xymatrix{#1 \ar[r]^-{#3} & #2 \\
#4 \ar@{|->}[r] & #5 \\}}

\newcommand{\mfa}{\mathfrak{a}}

\def\w{\wedge}

\def\im{\operatorname{im}}

\newcommand{\dm}{\operatorname{DM}}
\newcommand{\fold}{\operatorname{Fold}}

\newcommand{\Char}{\operatorname{char}}

\newcommand{\Id}{\textrm{id}}

\begin{document}
\title{Differential Modules with Complete Intersection Homology}



\author{Maya Banks}

\author{Keller VandeBogert}
\date{\today}

\maketitle

\begin{abstract}
    Differential modules are natural generalizations of complexes. In this paper, we study differential modules with complete intersection homology, comparing and contrasting the theory of these differential modules with that of the Koszul complex. We construct a Koszul differential module that directly generalizes the classical Koszul complex and investigate which properties of the Koszul complex can be generalized to this setting. 
\end{abstract}

\section{Introduction}

A \emph{differential module} is a module equipped with a square-zero endomorphism. While initially introduced at least as far back as the classical treatise of Cartan and Eilenberg \cite[Chapter 4]{cartan2016homological}, differential modules have become a topic of recent interest in commutative algebra motivated in part by their connections to the Buchsbaum--Eisenbud--Horrocks and Carlsson conjectures \cite{avramov2007class,10.1007/BFb0072815, buchsbaum1977algebra, HARTSHORNE1979117}, the BGG correspondence and Tate resolutions \cite{brown2021tate}, and the representation theory of algebras (\cite{rouquier2006representation}, \cite{ringel2017representations}). Differential modules are a natural generalization of chain complexes, and their study can thus provide a novel perspective on familiar objects such as free resolutions. More generally, there is an ever-expanding literature focusing on the use of differential modules to provide new insight on old conjectures (see for instance \cite{boocher2010rank}, \cite{csenturk2019carlsson}, and \cite{iyengar2018examples}), and also on the development of a general theory of differential modules for their own sake (see \cite{stai2017differential}, \cite{wei2015gorenstein}, and \cite{xu2015gorenstein}, and the references therein).

Our work is motivated in particular by recent work of Brown and Erman \cite{brown2021minimal}, in which they extend the notion of a \emph{minimal free resolution} to differential modules. Moreover, they prove a theorem indicating that the classical theory of minimal free resolutions still plays a significant role in understanding the structure of minimal free resolutions of differential modules. In particular, they show that for a differential module $D$ with homology $H(D)$, there is a \emph{free flag} $F$ whose structure and differential are partially controlled by the minimal free resolution of $H(D)$ and where there is a quasi-isomorphism $F\to D$ (see Theorem \ref{thm:danMikeThm} for the precise statement). This result begs the question of whether properties of the minimal free resolution of the homology $H(D)$ can be `lifted' to the free flag $F$. Brown and Erman examined this in the case where the homology is a Cohen-Macaulay codimension 2 algebra. We explore this question in the case where $H(D)$ is generated by a regular sequence.
    
In the classical theory of minimal free resolutions, the Koszul complex is one of the most fundamental objects of study for the simple reason of its sheer ubiquity; it is well-known to be a minimal free resolution of ideals generated by regular sequences, but even for non-complete intersections, properties of the Koszul homology of an ideal make their appearance in relation to DG-algebra techniques, the study of Rees algebras, and in the construction of more subtle types of complexes. In this paper, we begin developing a parallel theory for differential modules by first constructing a differential module analog of the Koszul complex, directly generalizing the classical case. We also investigate which properties of the Koszul complex lift to the minimal free resolution of differential modules whose homology is a complete intersection. We ask three main questions, the first of which is the following:

\begin{question}\label{q1}
For $R$ a graded-local ring, what are the differential modules $D$ whose homology is equal to the residue field $R/\m$? More broadly, can we classify the differential modules with homology $R/I$ a complete intersection?
\end{question}

\begin{example}\label{ex:smallkoszul}
Let $S = k[x_1,\ldots, x_n]$ for $k$ a field. Let $D = S^4$ with differential given by the matrix
\[
\begin{pmatrix}
0 & x_1 & x_2 & 0\\
0 & 0 & 0 & -x_2\\
0 & 0 & 0 & x_1\\
0 & 0 & 0 & 0\\
\end{pmatrix}.
\]
This differential module has homology $S/(x_1, x_2)$. In fact, it is the differential module we obtain by taking the Koszul complex on the regular sequence $(x_1, x_2)$ and viewing it as a differential module whose underlying module is the sum of the free modules in the Koszul complex and whose differential is the direct sum of the Koszul differentials. However, we can alter the differential by adding a nonzero entry to the top right corner without changing the homology of the differential module. That is, we get a family of differential modules $D_f$ with the same underlying module and differential given by 
\[
\begin{pmatrix}
0 & x_1 & x_2 & f\\
0 & 0 & 0 & -x_2\\
0 & 0 & 0 & x_1\\
0 & 0 & 0 & 0\\
\end{pmatrix}.
\]
Furthermore, by \cite{brown2021minimal}, every differential module with homology $S/(x_1, x_2)$ admits a quasi-isomorphism from some such $D_f$. However, not all choices of $f$ yield nonisomorphic differential modules. For instance for $f\in (x_1, x_2)$ we can perform row and column operations to show that $D_f$ is isomorphic to our original $D$, but this is not the case if (for instance) $f=1$.
\end{example}

In the above example, we see the structure of the Koszul complex itself mirrored in the structure of the differential modules $D_f$. This motivates our next two questions, which have to do with how much of this structure is actually preserved when we pass from resolutions to differential modules.

\begin{question}\label{q2}
Each Koszul complex corresponds to a single element in a particular exterior algebra. For a differential module $D$ with complete intersection homology, does every free flag resolution of the type described in \cite{brown2021minimal} via a similar construction?
\end{question}
\begin{question}\label{q3}
The classical Koszul complex is well-known to admit the structure of a DG-algebra. Does the generalization of the Koszul complex to differential modules admit any kind of analogous structure?
\end{question}

Our results give a partial answer to Question \ref{q1}---while we do not give a classification of 
\emph{all} differential modules with complete intersection homology, we do find constraints on such differential modules and prove results that simplify the classification question. We show that with some additional hypotheses, Question \ref{q2} can be answered affirmatively, but that absent these hypotheses we can construct examples of free flags with complete intersection homology that are not of the ``expected" form. In contrast to the previous two questions, the answer to Question \ref{q3} seems to be a resounding ``no", and indicates that generalizations of the classical notions of DG-algebra/module structures for differential modules will require much subtler formulations. 

\subsection{Results}
Let $R$ be a graded local ring, $D$ a module over $R$ and $d: D\to D$ an $R$-module endomorphism that squares to 0. We define the homology of $D$ to be $H(D) = \Ker(d)/\im{d}$. If the underlying module $D$ has the form $\bigoplus_{i\geq 0}F_i$ where $F_i$ is free and the differential $d$ satisfies that $d(F_j)\subseteq \bigoplus_{j<i}F_i$ then we call $D$ a \emph{free flag}. Note that any bounded below free complex can automatically be considered as a free flag.  A core result of \cite{brown2021minimal} states that any differential module $D$ admits a quasi-isomorphism from a free flag $(F, d)$ where $F = \bigoplus_{i\geq 0}F_i$ where 
\[
F_0\xleftarrow{\delta} F_1\xleftarrow{\delta} F_2 \leftarrow \cdots
\]
is a minimal free resolution of $H(D)$ and the $d$ restricts to $\delta$ when considered as a map $F_i\to F_{i-1}$. In this case, we can represent $d$ via a block matrix 
\[
\begin{pmatrix}
0 &\delta & A_{2,0} & A_{3,0} & \cdots & A_{m,0} \\
0 & 0 & \delta & A_{3,1} & \cdots & A_{m,1}\\
\vdots &  & & \ddots & & \vdots\\
0 & 0 & 0 & 0 &\cdots & A_{m,m-2}\\
0 & 0 & 0 & 0 & \cdots &\delta \\
0 & 0 & 0 & 0 & \cdots & 0   
\end{pmatrix}
\]
where $A_{i,j}: F_i\to F_j$. We refer to free flags of this type as \emph{anchored free flags}. Since every differential module admits an anchored free flag resolution, it is both convenient and reasonable to focus on differential modules of this form. To classify anchored free flag resolutions amounts to classifying the possibilities for $A_{i,j}$ that yield nonisomorphic differential modules. Our first result simplifies this task in the case where the homology is a complete intersection.

\begin{theorem}\label{thm:intro1}
Let $D$ be a free flag differential module with homology $H(D)$ a complete intersection where the differential $d$ is given by a block matrix as above. Let $F$ be the minimal free resolution of $H(D)$ considered as a differential module. Then $D\cong F$ if and only if $\im A_{i,0}\subseteq \im\delta$ for all $i\geq 2$.
\end{theorem}

This result says that, in the case of a homology induced free flag with complete intersection homology, the property of being isomorphic to a minimal free resolution of the homology can be detected by only the first row of blocks in the differential. This further tells us that we can at least partially characterize the differential modules with given complete intersection homology $R/I$ by the choices of nonzero maps $F_i\to R/I$.

We prove Theorem \ref{thm:intro1} in Section \ref{sec:freeFlagsAndFolds}, along with a discussion of minimal free resolutions of differential modules with homology $k$. In particular, we also show via explicit construction that the total Betti numbers of a differential module with homology $k$ may be strictly smaller than the sum of the Betti numbers of $k$.

In Section \ref{sec:genKoszulCx}, we show how graded commutative algebras admitting divided powers can be used to construct differential modules. This allows us to construct a Koszul differential module--a family of differential modules generalizing the Koszul complex (see Construction \ref{cons:genericKoszul}). In much the same way as the Koszul complex provides a valuable source of examples in the study of minimal free resolutions, Koszul differential modules generate a large set of examples of free flag differential modules. Moreover, we prove the following theorem which shows that in certain cases we can guarantee that anchored free flags with complete intersection homology are isomorphic to the Koszul differential module.

\begin{theorem}[See Theorem \ref{thm:nonCHar2} for the more general statement]\label{thm:intro2}
Let $R$ be a Noetherian graded local ring with maximal ideal $\m$, $D$ a differential $R$-module with $H(D)$ a complete intersection and let $F=\bigoplus_{i\geq 0}F_i\to D$ be an anchored free flag resolution with differential $d^F$. If the image of $d^F$ in $F_0$ generates a complete intersection contained in $\m$, then $F$ is isomorphic to a Koszul differential module. 
\end{theorem}

Finally, in Section \ref{sec:DGAlg} we consider the existence of DG-module structures on free flag resolutions as posited in Question \ref{q3}. This leads us to consider free flag resolutions that can be given the structure of a DG-module over the minimal free resolution of their homology. In the classical case of complexes, it is well-known that every free resolution admits the structure of a (possibly non-associative) DG-algebra structure, and hence the existence of such a structure is guaranteed. Our main result of this section is the following theorem which says that the existence of such a DG-module structure on an arbitrary free flag resolution is in fact a much rarer property. 

\begin{theorem}\label{thm:intro3}
Let $F$ be an anchored free flag with complete intersection homology. If $F$ admits the structure of a DG-module over the minimal free resolution of $H(F)$, then $F$ is isomorphic to the Koszul complex considered as a differential module.
\end{theorem}

This theorem gives an example of a property of free resolutions that does \emph{not} generalize to the setting of differential modules. The inability to generalize this property tells us that the DG-algebra structure of free resolutions, at least for complete intersections, relies on structure that is unique to free resolutions rather than structure that can be extended to free flags. This is in contrast with properties that are successfully generalized to free flags in \cite{avramov2007class} and \cite{brown2021minimal}.

\subsection*{Acknowledgements} We thank Daniel Erman and Michael Brown for many helpful conversations and comments throughout all stages of our work.

\section{Background}\label{sec:background}

In this section, we introduce some background and notation on differential modules that will be used throughout the paper. This includes a straightforward reformulation of free flag differential modules without reference to matrices, which will be useful for avoiding complicated matrix computations; this formulation is implicit in the work of Avramov, Buchweitz, and Iyengar \cite{avramov2007class}, but we state it here since it will be used frequently throughout the paper. 


\begin{definition}\label{def:diffMod}
A \emph{differential module} $(D,d)$ or $(D,d^D)$ is an $R$-module $D$ equipped with an $R$-endomorphism $d = d^D \colon D \to D$ that squares to $0$. A differential module is $\bbz^n$-graded of multidegree $\mathbf{a}$ if $D$ is equipped with a $\bbz^n$ grading over $R$ such that $d \colon D \to D(\mathbf{a})$ is a graded map. The category of degree $\mathbf{a}$ differential $R$-modules will be denoted $\dm (a,R)$; when $a = 0$, the more concise notation $\dm (R)$ will be used.

The \emph{homology} of a differential module $(D,d)$ is defined to be $\ker (d) / \im (d)$. If $D$ is $\bbz^n$-graded of degree $\mathbf{a}$, then the homology is defined to be the quotient $\ker(d) / \im(d(- \mathbf{a}))$.

A differential module is \emph{free} if the underlying module $D$ is a free $R$-module. If $(R,\m,k)$ is a local ring or $R = k[x_1, \dots , x_n]$ (equipped with some positive grading) and $D$ is a graded $R$-module, then $D$ is \emph{minimal} if $d \otimes k = 0$.

A morphism of differential modules $\phi \colon (D,d^D) \to (D', d^{D'})$ is a morphism of $R$-modules $D \to D'$ satisfying $d^{D'} \circ \phi = \phi \circ d^D$. Notice that morphisms of differential modules induce well-defined maps on homology in an identical fashion to the case of complexes. A morphism of differential modules is a \emph{quasi-isomorphism} if the induced map on homology is a quasi-isomorphism. 
\end{definition}

\begin{remark}
The collection of differential modules and their morphisms forms a category, denoted $\dm (R)$. Notice that a differential $R$-module $(D,d)$ is equivalently a module over the ring $R[x]/(x^2)$, as mentioned in the introduction. In particular, the category $\dm (R)$ is equivalently the category of $R[x]/(x^2)$-modules, and as such $\dm (R)$ is an abelian category. The category $\dm (a,R)$ is also equivalently graded modules over $R[x]/(x^2)$ in the case that $x$ is given degree $a$.
\end{remark}

The following definition will play an essential role throughout the paper, and it allows us to view the category of complexes as a subcategory of the category of differential modules. 

\begin{definition}\label{def:fold}
Given any complex $F$, there is a functor
\begingroup\allowdisplaybreaks
\begin{align*}
    \fold \colon \textrm{Com} (R) & \to \dm (R), \\
    (F, d^F) &\mapsto \Big( \bigoplus_{i\in \bbz} F_i , \bigoplus_{i \in \bbz} d^F_i \Big), \\
    \{ \phi \colon F_\bullet \to G_\bullet \}   &\mapsto \Big\{ \bigoplus_{i \in \bbz} \phi_i \colon \fold (F_\bullet) \to \fold (G_\bullet) \Big\}.
\end{align*}
\endgroup
The object $\fold (F_\bullet)$ will often be referred to as the \emph{fold} of the complex $F_\bullet$.
\end{definition}

The following definition introduces \emph{free flags}. These are a proper subclass of differential modules that still generalize complexes of free modules, and in general are better behaved than arbitrary differential modules. One way to think of free flags is as differential modules admitting a finite length filtration whose associated graded pieces are themselves free $R$-modules.

\begin{definition}\label{def:flagGrading}
Let $D$ be a differential module. Then $D$ is a \emph{free flag} if $D$ admits a splitting $D = \bigoplus_{i \in \bbz} F_i$, where each $F_i$ is a free $R$-module, $F_i = 0$ for $i \leq 0$, and $d_D (F_i) \subseteq \bigoplus_{j < i} F_j$. 

Given a free flag $D$ with associated splitting $D = \bigoplus_{i \in \bbz} F_i$, define $D^i \colon= \bigoplus_{j \leq i} F_j$. This will be referred to as the \emph{flag grading} on $F$. By definition of a free flag, one has $d_D (D^i) \subset D^{i-1}$, implying that $D$ may be viewed as a complex with homological grading induced by the flag grading.

Associated to a free flag $D$, there are maps $A_{i,j} \colon F_i \to F_j$ induced by splitting the maps $d_D \colon F_i \to D^{i-1}$ with the isomorphism $\hom (F_i , D^{i-1} ) = \bigoplus_{j < i} \hom(F_i , F_j)$. 
\end{definition}

A core theme of \cite{avramov2007class} is that the flag structure on a differential module allows many proofs from classical homological algebra to be generalized to differential modules by substituting the homological grading for the flag grading. As such, it is useful to pass from general differential modules to free flags, which we can do using the following definition.

\begin{definition}
For any other differential module $D$, a \emph{free flag resolution} $F$ of $D$ is a free flag $F$ equipped with a quasi-isomorphism $F \to D$. A \emph{minimal free resolution} of $D$ is a quasi-isomorphism $M\to D$ that factors through a free flag resolution $F$ such that $M\to F$ is a split injection and $M$ is minimal.
\end{definition}

\begin{remark}
Notice that if $F_\bullet \to M$ is a minimal free resolution of a module $M$, then $\fold (F_\bullet) \to M$ (where $M$ is viewed as having the $0$ endomorphism) is a minimal free flag resolution of $M$ (since $H(\fold (F_\bullet)) = \bigoplus_{i \in \bbz} H_i (F_\bullet)$). In general, there may be free flag resolutions of $M$ that do not arise as the fold of a complex, and it is an interesting question as to when a free flag is isomorphic to the fold of some complex. We give a characterization of this property for certain classes of free flag resolutions in Section 3 which turns out to be quite effective in proving some general statements about free flag resolutions.
\end{remark}

One way to think of free flags is as strictly upper triangular block matrices (as in the setup to Theorem \ref{thm:intro1}). This can be useful for explicit computations and more matrix-theoretic methods, but we will also find it useful to think of free flags in a way that is not reliant on matrices. The following observation is a coordinate-free reformulation of the definition of a free flag that highlights the data of the  maps $A_{i,j} \colon F_i \to F_j$ which determine the flag. 

\begin{obs}\label{obs:DMsCoordFree}
A free flag is equivalently the data of a collection of free modules $\{ F_i \mid i \in \bbz \}$ and maps $\{ A_{i,j} \colon F_i \to F_j \mid j<i \}$, such that for all $j<i$, one has the relation
$$\sum_{j < k < i} A_{k,j} A_{i,k} =0.$$
\end{obs}

\begin{remark}
In order to distinguish the structure maps $A_{i,j}$ given in Observation \ref{obs:DMsCoordFree}, we will often use the more precise notation $A_{i,j}^D$ to specify that these maps determine the differential module $D$.
\end{remark}

The theory of minimal free resolutions of arbitrary differential modules turns out to be quite subtle. However, the following result of Brown and Erman shows that the classical theory of minimal free resolutions of modules still plays an important role in understanding the homological properties of differential modules.

\begin{theorem}[{\cite[Theorem 3.2]{brown2021minimal}}]\label{thm:danMikeThm}
Let $D$ be a differential module and $(F_\bullet , d) \to H(D)$ a minimal free resolution of $H(D)$. Then $D$ admits a free flag resolution where the underlying free module is $F_\bullet$ and where, in the notation of Observation \ref{obs:DMsCoordFree}, $A_{i,i-1} = d_i$ for all $i$.
\end{theorem}

\section{Anchored Free Flags and Folds of Complexes}\label{sec:freeFlagsAndFolds}

In this section, we prove Theorem \ref{cor:CIequiv}, which is our first structural result about free flag resolutions whose homology is a complete intersection.  Specifically, we consider a free flag resolution $F$ as in Theorem \ref{thm:danMikeThm} whose homology is a complete intersection, and we show that question of whether or not $F$ is trivial—i.e. where $F$ is isomorphic to the fold of a Koszul complex—is entirely determined by an analysis of the “top row” of the differential. As an application, we then completely classify all differential modules $D$ where $H(D)$ is isomorphic to the residue field $k$.


We conclude the section with some interesting examples illustrating the subtlety of minimal free resolutions of differential modules. In particular, we show that if $R = k[x_1 , \dots , x_n]$ is a standard graded polynomial ring over a field, then $k$ viewed as a differential $R$-module in degree $2$ has total Betti numbers \emph{strictly} less than the total Betti numbers of $k$ when viewed as an $R$-module in the usual fashion.

\begin{definition}\label{setup:DGsetup}
Let $D$ be a free flag and $(F_\bullet , d) \to H(D)$ a minimal free resolution of $H(D)$. If $D$ arises as in the statement of Theorem \ref{thm:danMikeThm}, then $D$ will be called an \emph{anchored free flag resolution}.
\end{definition}

\begin{remark}
An anchored free flag resolution has differential with off-diagonal blocks coming from the minimal free resolution of the homology. These off-diagonal blocks can be thought of as ``anchors" for the maps $A_{i,j}$ for $i- j \geq 2$; more precisely, we have complete freedom to choose the `higher-up" maps $A_{i,j}$ up to the constraint that these maps must still make the corresponding differential square to $0$. 
\end{remark}

Conceptually, the following lemma shows that if one can perform column operations on the matrix representation of the differential of an anchored free flag to cancel a term $A_{i,0}$, then one can in fact perform column operations to cancel all other terms appearing along the associated diagonal.


\begin{lemma}\label{lem:MayaLem}
Let $D$ be an anchored free flag and assume $\im A_{i,0}^D \subset \im d_1$ for all $2 \leq i \leq m$ for some given $m$. Then $D$ is isomorphic to a differential module $D'$ satisfying $A_{i, \ell}^{D'} = 0$ for all $2 \leq i - \ell \leq m$. 
\end{lemma}

\begin{proof}
The assumption $\im A_{i,0}^D \subset \im d_1$ for all $2 \leq i \leq m$ implies that $D$ is isomorphic to a differential module $D'$ with the same underlying module determined by maps $\{ A_{i,j}^{D'} \colon F_i \to F_j \mid i<j \}$, but satisfying $A_{i,0}^{D'} = 0$ for all $i \leq m$ and $A_{i,i-1}^{D'} = d_i$ for all $i$. 

Let $2 \leq i-1 \leq m$. Then there is the relation
$$\sum_{j<i} A_{j,0}^{D'} \cdot A_{i,j}^{D'} = 0.$$
Since $i \leq m+1$ and $j <i$, one has that $A_{j,0}^{D'} = 0$ for each $j > 1$ appearing in the above equality. Thus $d_1 (A_{i,1}^{D'}) = 0$, and exactness implies that $ \im A_{i,1} \in \im d_2$ for each $i \leq m+1$. Replacing $D$ with $D'$ and iterating this argument, the result follows.
\end{proof}

In particular, the above gives a criterion for $D$ to be isomorphic to the fold of the minimal free resolution of its homology. One might hope that this is in fact an equivalence---that is, that any differential module isomorphic to the fold of the minimal free resolution of its homology can be identified in this way. This is in general not the case, as we see in the following example.

\begin{example}\label{ex:cantCancel}
Let $R = k[x_1,x_2]$ and $E$ be a rank $2$ free module on the basis $e_1 , e_2$. Let $D = \bigwedge^\bullet E$ be the free flag with differential
$$\begin{pmatrix} 0 & x_1^2 & x_1x_2 & x_1 \\
0 & 0 & 0 & -x_2 \\
0 & 0 & 0 & x_1 \\
0 & 0 & 0 & 0  \\
\end{pmatrix}.$$
Let $K_\bullet$ denote the minimal free resolution of $H(D)$
\[
K_\bullet = \quad \bigwedge^0 E\xleftarrow{\begin{pmatrix}x_1^2 & x_1x_2\end{pmatrix}} \bigwedge^1 E\xleftarrow{\begin{pmatrix}
-x_2\\x_1\end{pmatrix}} \bigwedge^2 E
\]
 Then the morphism of differential modules $D\to \fold(D_\bullet)$ induced by
\begingroup\allowdisplaybreaks
\begin{align*}
    1 & \mapsto 1, &
    e_1 \mapsto e_1, & \quad e_2 \mapsto e_2 , &
    e_1 \w e_2 & \mapsto e_1 \w e_2 + e_1,
\end{align*}
\endgroup
is an isomorphism, even though $x_1 \notin (x_1^2 , x_1 x_2) = \im d_1^K$. 
\end{example}

Note that the isomorphism described in the above example corresponds to performing \emph{row} operations on the differential to cancel out the $x_1$ in the corner. In this case, we are able to cancel via row operations but not by column operations. This is explained by the lack of symmetry in the differentials of the minimal free resolution of $H(D)$. In fact, when the minimal free resolution of $H(D)$ is given by the Koszul complex---i.e. when $H(D)$ is a complete intersection---this scenario cannot arise, as we will show next.

First, notice that any morphism of free flags $\phi \colon (D,d) \to (D , d')$ of free flags with the same underlying free module $\bigoplus_{i \in \bbz} F_i$ splits as a direct sum of maps $\phi_{i,j} \colon F_i \to F_j$ for every $j \leq i$. We will employ this notation in the following statement:

\begin{lemma}\label{lem:theConverse}
Let $D$ be an anchored free flag. Assume that there exists an isomorphism $\phi \colon D \to \fold (F_\bullet)$ for some complex $F_\bullet$, and assume that $\phi_{i-1,0} (d_i^F (F_i)) \subset \im d_1^F$ for all $i \geq 2$. Then $\im A_{i,0} \subset \im d_1^F$ for all $i \geq 2$.
\end{lemma}

\begin{proof}
View the free flags $D$ and $\fold (F_\bullet)$ as complexes with $D^i$ in homological degree $i$, and the differentials induced by $d^D$ and $d^F$, respectively. Then the assumption $D \cong \fold (F_\bullet)$ (as differential modules) is equivalent to assuming that there is an isomorphism of \emph{complexes} $D \cong \fold (F_\bullet)$ (where the flag grading induces the homological degree, as in Definition \ref{def:flagGrading}). In each homological degree there is a splitting $D^i = D^{i-1} \oplus F_i$ and $(d_D)_i = (d_D)_{i-1} \oplus (d_D)|_{F_i}$, so it suffices to compute the map $\phi \colon F_i \to D^{i-1}$. Observe that $\phi$ splits as a sum of morphisms:
$$\phi = \sum_{j\leq i} \phi_{i,j},$$
where $\phi_{i,j} \colon F_i \to F_j$. Moreover, changing bases of $F_0$ as necessary, $\phi|_{F_0}$ may be chosen to be the identity. Let $f_i \in F_i$ viewed as an element in homological degree $i$; the assumption that $\phi$ is a morphism of complexes then yields:
\begingroup\allowdisplaybreaks
\begin{align*}
    \phi (d^D (f_i) ) &= \phi \Big( \sum_{0<j<i} A_{i,j} (f_i) \Big) + A_{i,0} (f_i) \\
    &= \sum_{j \leq i} d_j^F (\phi_{i,j} (f_i)) \\
    &= d^F (\phi( f_i)) . 
\end{align*}
\endgroup
Comparing the above restricted to the direct summand $F_0$, one obtains the equality
\begin{equation}\label{eqn:deg0equality} d_1^F ( \phi_{i,1} (f_i) ) = A_{i,0} (f_i) + \sum_{0 < j < i} \phi_{j,0} (A_{i,j} (f_i)).\end{equation}
Now, we proceed by induction on $i$ to prove the desired statement. When $i=2$, the above equality becomes
$$d_1^F (\phi_{2,1} (f_2)) = A_{2,0} (f_2) + \phi_{1,0} ( d_2^F (f_2)).$$
The assumption that $\phi_{1,0} ( d_2^F (f_2)) \in \im d_1^F$ implies that $A_{2,0} (F_2) \subset \im d_1^F$, and Lemma \ref{lem:MayaLem} implies that $D$ may be replaced with a differential module satisfying $A_{i,i-2} = 0$ for all $i \geq 2$. Proceeding inductively, assume $i >2$; by induction, we may assume that $A_{j,k} = 0$ for all $1<j-k <i$. The equality \ref{eqn:deg0equality} then reduces to 
$$d_1^F (\phi_{i,1} (f_i) ) = A_{i,0} (f_i) + \phi_{i-1,0} (d_i^F (f_i)),$$
and again the assumption $\phi_{i-1,0} ( d_i^F (f_i)) \in \im d_1^F$ implies that $A_{i,0} (f_i) \in \im d_1^F$, and Lemma \ref{lem:MayaLem} allows us to replace $D$ with a differential module satisfying $A_{j,k} = 0$ for all $j-k \leq i$. Iterating this argument, the result follows.
\end{proof}

The above holds, in particular, when $H(D)$ is a complete intersection.

\begin{theorem}\label{cor:CIequiv}
Let $D$ be an anchored free flag and assume that $H(D)$ is a complete intersection. Then
$$D \cong \fold (F_\bullet) \quad \iff \quad \im A^D_{i,0} \subset \im d_1 \ \textrm{for all} \ i \geq 2.$$
\end{theorem}

\begin{proof}
Let $H(D) = R/\mfa$ where $\mfa$ is generated by a regular sequence. By definition of the Koszul complex, the minimal free resolution of $R/ \mfa$ satisfies $d_i^F \otimes R/\mfa = 0$ for all $i$, whence the assumption $\phi_{i-1,0} (d_i^F (F_i)) \subset \im d_1^F$ of Lemma \ref{lem:theConverse} is trivially satisfied.
\end{proof}

Under the perspective of free flags as upper triangular block matrices, this says that the property of $D$ being isomorphic to the resolution of its homology is completely detectable via only the top row of the matrix. On the other hand, in the case where $D$ is graded, the degrees of the entries of the top row may be deduced by using this grading. Adding this additional structure gives strong restrictions on the possibilities for differential modules $D$ with complete intersection homology. In the most restrictive case, when $H(D)\cong k$, we have the following.

\begin{cor}\label{cor:koszulDegs}
Let $D\in \dm(R,a)$ be an anchored free flag with $H(D) \cong k$. If $a\neq 2$, then $D \cong \fold (K_\bullet)$, where $K_\bullet$ denotes the Koszul complex resolving $k$.
\end{cor}

\begin{proof}
Since $D$ has degree $a$, the minimal free resolution of $H(D)$ is given by the Koszul complex $K_\bullet$ with the $i^{th}$ free module twisted by $ia$. The maps $A_{i,j}\colon R(ia-a)\to R(ja-j+a)$ in the differential on $D$ therefore have degree $(ja-j+a)-(ia-i)$. When $\deg A_{i,0}\neq 0$, we have $\im A_{i,0}\otimes k = 0$. On the other hand, $(ja-j+a)-(ia-i) = 0$ only when $a=2$ and $i-j=2$. Thus for $a\neq 2$ the result follows from Corollary \ref{cor:CIequiv}.
\end{proof}

The above corollary implies that in degree $a\neq 2$, all differential $R$-modules with homology $k$ have isomorphic anchored free flag resolutions, and that furthermore this resolution is minimal and isomorphic to the Koszul complex. However, this is not true for differential modules of degree $2$. In general, any $R$-module may be viewed as a degree $a$ differential module, for any integer $a$, and the following result shows that the homological invariants of $M \in \dm (a, R)$ can vary as the degree $a$ varies.

\begin{prop}\label{prop:smallrank}
Let $S := k[x_1, \ldots, x_n]$ be a standard graded polynomial ring over a field $k$, where $n\geq 2$. Then there exists a degree $2$ free differential module $D$ of rank $2^{n-1}$ with $H(D) \cong k$.
\end{prop}

\begin{proof}
We construct the differential module inductively. For $n=2$, let $D = S^2$ with differential given by the matrix $\begin{pmatrix}
x_1x_2 & -x_2^2\\
x_1^2 & -x_1x_2
\end{pmatrix}$. One can check that $D$ is a degree $2$ differential module, and moreover that $H(D) \cong k$. 

Now let $(F,d)$ be a free differential module of rank $2^{n-2}$ over $R = k[x_1, \ldots, x_{n-1}]$. Let $S = k[x_1, \ldots, x_n]$ and $F' = (F\otimes_R S)\oplus (F\otimes_R S(1))$, take $d'\colon F'\to F'(2)$ to be the map given by the block matrix $\begin{pmatrix} -d\otimes_R S & x_n\Id\\ 0 & d\otimes_R S\end{pmatrix}$ (note that $F'$ is the mapping cone of the map $F\otimes_R S\to F\otimes_R S$ given by multiplication by $x_n$). The matrix $d'$ squares to 0, so $F'$ is a rank $2^{n-1}$ differential module of degree 2. We claim that $H(F)\simeq H(F')$. 

We can compute the homology of $F'$ directly by computing the kernel and image of $d'$, denoted $Z$ and $B$, respectively. A priori, 
\[ B = \left\langle (df, 0), (x_nf, df) ~|~ f\in F (-2)\right\rangle,\] and
\[
Z = \left\langle    (g_1, g_2) ~|~ dg_2 = 0 \text{ and } dg_1 = x_ng_2   \right\rangle.
\]

We claim that $Z$ is actually generated by elements of the form $(x_nf, df)$ and $(h, 0)$ where $h\in \ker(d)$ and $f\in F (-2)$. It is evident that elements of this form are in $Z$, so suppose that $(g_1, g_2)\in Z$; in particular, $dg_1 = x_ng_2$. We can write $g_1$ as $x_nf+h$ where $x_n$ does not divide any component of $h$. Since $d$ is a matrix over $R$, the variable $x_n$ also does not divide any component of $dh$. But $dg_1 = dx_nf + dh = x_ndf+dh = x_ng_2$, which means $dh = 0$ and $x_ndf = x_ng_2$, so $g_2 = df$. It follows that $(g_1, g_2) = (x_nf, df) + (h,0)$ for $h\in \ker(d)$, so $Z$ is generated by elements of the form $(x_nf, df)$ and $(h,0)$ for $f\in F$ and $h\in \ker(d)$.

Combining all of the above information, one finds:
\begingroup\allowdisplaybreaks
\begin{align*}
H(F') = \frac{Z}{B} &= \frac{\left\langle (x_nf, df), (h,0) ~|~ f\in F(-2), h\in\ker(d)  \right\rangle}{\left\langle (df, 0), (x_nf, df) ~|~ f\in F(-2)\right\rangle} \\
&= \frac{\langle (h,0) ~|~ h\in\ker(d)\rangle}{\langle (df, 0) ~|~ f\in F(-2)\rangle} \\
&\simeq \frac{\ker(d)}{\im(d(-2))} = H(F).
\end{align*}
\endgroup

Thus using our example for $n=2$, a rank $2^{n-1}$ free differential module with homology $k$ can be constructed for any $n\geq 2$.
\end{proof}


If we define the \emph{Betti numbers} of a differential module as in \cite{brown2021minimal}, then the sum of the Betti numbers of $D$ is equal to the rank of the minimal free resolution of $D$. We next show that the differential module constructed in \ref{prop:smallrank} is actually a minimal free resolution. This will imply that the sum of the betti numbers of a degree 2 differential module with homology $k$ may be at least as small as $2^{n-1}$, strictly smaller than the total rank of a minimal free resolution with the same homology. Although this differential module is certainly free and minimal, it is not immediately obvious that it is a minimal free \emph{resolution}, since this requires it to be a summand of a free flag resolution. To prove that this is indeed the case, we leverage the mapping cone structure of the differential module constructed in \ref{prop:smallrank}.

\begin{lemma}
Let $F$ be a free differential $S$-module of degree $a$ and let $\phi\colon F\to F$ be the map given by multiplication by $f$. Suppose as in Proposition 4.1 of \cite{brown2021minimal} that $F = M\oplus T$, where $M$ is minimal and $T$ is contractible. Then 
\[\cone(\phi)\simeq \cone(\phi|_M)\oplus T'\] 
where $T'$ is contractible.
\end{lemma}

\begin{proof}
Written as a block matrix, $\phi\colon M\oplus T\to M\oplus T$ has the form $\begin{pmatrix}
f & 0\\ 0 & f
\end{pmatrix}$. By definition, $ \cone(\phi) = F\oplus F(a) = M\oplus T \oplus M(a)\oplus T(a)$ with differential given by the block matrix
\[
d = \begin{pmatrix}
-d_M & 0 & f & 0\\
0 & -d_T & 0 & f\\
0 & 0 & d_M & 0\\
0 & 0 & 0 & d_T
\end{pmatrix}
\]

We can perform row and column operations to see that this is isomorphic to the differential module $M\oplus M(a)\oplus T\oplus T(a)$ with differential 
\[
d' = \begin{pmatrix}
-d_M & f & 0 & 0\\
0 & d_M & 0 & 0\\
0 & 0 & -d_T & 0\\
0 & 0 & 0 & d_T
\end{pmatrix}
\]
so $\cone(\phi)\simeq \cone(\phi|_M)\oplus T'$ where $T' \simeq T\oplus T(a)$ is contractible.

\end{proof}



\begin{cor}
Let $D_n$ be the rank $2^{n-1}$ free differential module over $k[x_1, \ldots, x_n]$ defined in Proposition \ref{prop:smallrank}. Then $D_n$ is a minimal free resolution of a differential module with homology $k$.
\end{cor}

\begin{proof}
Proceed by induction on $n$. For $n=2$, the module $S\oplus S$ with differential $\begin{pmatrix} xy & x^2\\ -y^2 & xy\end{pmatrix}$ is its own minimal free resolution (See \cite{brown2021minimal}, Example 5.8). In particular, it is a minimal free summand of the free $S$-module $S\oplus S(1)^2 \oplus S(2)$ with differential 
\[
\begin{pmatrix}
0 & x & y & 1\\
0 & 0 & 0 & -y\\
0 & 0 & 0 & x\\
0 & 0 & 0 & 0
\end{pmatrix}.
\]

On the other hand, in degrees $a\neq 2$, Corollary \ref{cor:koszulDegs} tells us that any differential module $D$ with $H(D)\cong k$ has as its minimal free resolution the fold of the Koszul complex resolving $k$, which has rank $2^n$. This tells us that the Betti numbers of a differential module over $k[x_1, \ldots, x_n]$ with homology $k$ must be at least $2^n$ in degree $a\neq2$ but may be as small as $2^{n-1}$ in degree $2$.

Note that this is a degree 2 differential module whose underlying module is isomorphic to that of $\fold(K_\bullet)$, where $K_\bullet$ is the Koszul complex on two variables, twisted to be homogeneous degree $2$ but whose underlying differential is $d_K + \alpha$ where $\alpha$ is a block matrix whose nonzero entries lie in blocks strictly above those of $d_K$.

Now assume that $D_{n-1}$ is a minimal free resolution. Since its homology is $k$, it is a summand of a free flag resolution $F_{n-1}$ whose underlying module is isomorphic to $\fold(K_\bullet)$, the Koszul complex on $n-1$ variables, with differential $d_K + \alpha$ for some map $\alpha$ such that when considered as a block matrix, all nonzero blocks lie strictly above the nonzero blocks of $d_F$. Let $f\colon F_{n-1}\otimes S\to F_{n-1}\otimes S$ be multiplication by $x_n$. Taking the cone of $f$, we have $\cone(f)\simeq D_n\oplus T_n$ where $D_n\simeq \cone(f|_{D_{n-1}})$. On the other hand, the underlying module $\cone(f)$ is isomorphic to that of $\cone(\fold(K_\bullet))$ with differential $\begin{pmatrix}
-d_K-\alpha & x_n\\ 0 & d_K+\alpha
\end{pmatrix}$. If we let $K'_\bullet$ be the Koszul complex on $n$ variables, we see that the underlying module of $\cone(f)$ is thus isomorphic to that of $\fold(K'_\bullet)$ with differential $d_{K'}+\alpha'$ where $\alpha'$ again has the property that when considered as a block matrix all nonzero blocks lie above the nonzero blocks of $d_{K'_\bullet}$. In other words, $\cone(f)$ is a free flag as in Theorem \ref{thm:danMikeThm} that splits as a direct sum of $\cone(f|_{D_{n-1}})\simeq D_n$ with a contractible differential module. By definition, the differential module $D_n$ is a minimal free resolution of $k$ when viewed as an object in $\dm (2, R)$. 
\end{proof}

Note that because the differential module $D$ constructed in Proposition \ref{prop:smallrank} is not a free flag, it does not itself contradict the (disproven) conjecture of Avramov--Buchweitz--Iyengar that the rank of a free flag over $R$ with finite length homology is at least $2^{\dim(R)}$ \cite[Conj. 5.3]{avramov2007class}. In fact, the example when $n=2$ appears in \cite{avramov2007class}, and this construction directly generalizes their example.

\section{A Generalization of the Koszul Complex for Differential Modules}\label{sec:genKoszulCx}

In this section, we introduce a differential module analog of the Koszul complex and show that under certain hypotheses, all anchored free flags are isomorphic to this Koszul differential module. We also provide an example showing that in general not all such free flags with complete intersection homology are obtained by this Koszul differential module; to begin this section, we recall the definition of the Koszul complex that will be most convenient for our purposes.

Let $R$ be a ring and $E$ be a free $R$-module on basis $e_1 , \dots , e_n$ and $\psi \colon E \to R$ any $R$-module homomorphism. The notation $e_I$ will be shorthand for the basis element $e_{i_1} \w \cdots \w e_{i_k}$, where $I = \{ i_1 < \cdots < i_k\}$ is an indexing set of the appropriate size. Recall that the Koszul complex can be constructed as the complex with the $i$th exterior power $\bigwedge^i E$ sitting in homological degree $i$ and differential
$$e_{j_1 , \dots , j_i} \mapsto \sum_{k=1}^i (-1)^{k+1} \psi (e_{j_k}) e_{j_1 , \dots , \widetilde{j_k} , \dots , j_i}.$$
Another way to view this map is as the composition
\begingroup\allowdisplaybreaks
\begin{align*}
    \bigwedge^i E &\xrightarrow{\textrm{comultiplication}} E \otimes \bigwedge^{i-1} E \\
    &\xrightarrow{\psi \otimes 1} R \otimes \bigwedge^{i-1} E \cong \bigwedge^{i-1} E.
\end{align*}
\endgroup
This is equivalently described as multiplication by an element $f \in E^*$ (recall that $\bigwedge^\bullet E$ is a graded $\bigwedge^\bullet E^*$-module and vice versa). Let $A_{i,0} \colon \bigwedge^i E \to R$. Notice that such a map is equivalently induced by multiplication by an element $f_i \in \bigwedge^i E^*$. This is because the pairing $\bigwedge^i E \otimes \bigwedge^{n-i} E \to \bigwedge^n E$ is perfect.

Our goal will now be to generalize the exterior algebra structure of the Koszul complex to define a ``Koszul differential module". We illustrate with an example

\begin{example}\label{ex:koszul}
Let $R = k[x_I \mid I \subset [4], \ I \neq \varnothing]$ and let $f_i$ for $i = 1 , \dots , 4$ be the elements of $\bigwedge^i E^*$ induced by the maps
\begingroup\allowdisplaybreaks
\begin{align*}
     E &\to R &\quad \bigwedge^2 E &\to R &\quad \bigwedge^3 E &\to R &\quad \bigwedge^4 E &\to R \\
     e_i & \mapsto x_i &\quad
     e_{ij} &\mapsto \begin{cases}
     0 & \textrm{if} \ i=1,\\
     x_{ij} & \textrm{otherwise}.
     \end{cases} &\quad
     e_{ijk} &\mapsto x_{ijk} &\quad
     e_{1234} &\mapsto x_{1234}.
\end{align*}
\endgroup
Define a free flag $F$ whose underlying module is $\bigwedge^\bullet E$ and whose differential is given by the maps $A_{i,j}:\bigwedge^iE\to\bigwedge^jE$ defined by $A_{i,j}(g) = (-1)^{ij}f_{i-j}g$. To check that this indeed defines a differential module amounts to checking that for each $i>j$
\[
\sum_{j < k < i} A_{k,j} A_{i,k} = \sum_{j<k<i} (-1)^{kj}(-1)^{ik}f_{k-j}f_{i-k} = 0
\]

In this case, one just needs to verify the relations
\[
f_1^2=0, \quad f_1f_2-f_2f_1 = 0, \quad f_1f_3+f_3f_1+f_2^2 = 0.
\]

\end{example}

We can generalize the construction in Example \ref{ex:koszul} to obtain a differential module in a similar way given a suitably chosen algebra. Recall that an algebra \emph{admits divided powers} if the subalgebra generated by elements of even degree satisfies the axioms of a divided power algebra. A canonical example of such an algebra to keep in mind is the exterior algebra on a free module, where the elements of even degree are the divided power elements.

\begin{prop}\label{prop:genKoszul}
Let $T$ denote any graded-commutative $R$-algebra admitting divided powers and with the structure of a graded $T^*$-module (where $T^*$ denotes the graded dual). Given any $f_i \in T_i^*$, the notation $f_i \colon T_\ell \to T_{\ell-i}$ will denote the left-multiplication map. Assume either:
\begin{enumerate}
    \item $\Char R = 2$, or
    \item $\Char R \neq 2$ and $f_i \cdot f_j = 0$ if both $i$ and $j$ are even.
\end{enumerate}
Define $A_{i,j} := (-1)^{ij} f_{i-j} \colon T_i \to T_j$. Then the data 
$$\{ T_i , \  A_{i,j} \colon T_i \to T_j \mid j<i, \ i \in \bbz \}$$
determines a differential module.
\end{prop}


\begin{proof}
One only needs to verify that $\sum_{j < k < i} A_{k,j} A_{i,k} = 0$ for all choices of $i$ and $j$, and this is a straightforward computation. Assume that $i+j$ is odd; one computes:
\begingroup\allowdisplaybreaks
\begin{align*}
    \sum_{k=j+1}^{i-1} A_{k,j} \cdot A_{i,k} &= \sum_{k=j+1}^{(i+j-1)/2} (A_{k,j} A_{i,k} + A_{i+j-k,j} A_{i,i+j-k} ) \\
    &= \sum_{k=j+1}^{(i+j-1)/2} \Big( (-1)^{jk + ik} + (-1)^{j(i+j-k)+i(i+j-k)+(i-k)(k-j)} \Big) f_{k-j} f_{i-k}. 
\end{align*}
\endgroup
Thus it suffices to show that the coefficient 
$$(-1)^{jk + ik} + (-1)^{j(i+j-k)+i(i+j-k)+(i-k)(k-j)}$$
is $0$ if $i-k$ or $k-j$ is odd. Since the above expression is symmetric in $i$ and $j$ modulo $2$, it is of no loss of generality to assume that $i \cong_2 k+1$. One computes:
$$jk+ik \equiv_2 jk, \quad \textrm{and}$$
$$j(i+j-k)+i(i+j-k)+(i-k)(k-j) \equiv_2 j(j+1) + (k+1)(j+1) + j+k \equiv_2 jk+1.$$
Thus the coefficient is $0$ if $i-k$ and $k-j$ are not both even, and if they are both even, then $f_{i-k} f_{k-j} = 0$ or appears with coefficient $2$, implying that these terms vanish as well. 

If $i+j$ is even, then the computation is identical, with the only difference being that the term $f_{(i+j)/2}^2$ appears. If $\Char k \neq 2$, then this term vanishes by assumption. If $\Char k =2$, then the assumption that $T$ admits divided powers implies that $f_{(i+j)/2}^2 = 2 f_{(i+j)/2}^{(2)} = 0$, so all terms again vanish.
\end{proof}

 \begin{notation}
Let $f_i \colon T_i \to R$ for $i=1 , \dots , n$ be a collection of maps induced by multiplication by $f_i \in (T_i)^*$, where $T_\bullet$ is any graded-commutative algebra. The notation $K(f_1 , \dots , f_n)$ will denote the (not necessarily differential) module induced by the data $\{ T_i , \ (-1)^{ij} f_{i-j} \}$. 
\end{notation}


We can also construct a ``generic" Koszul differential module; we will see that the theorems below give criteria for which free flags of the appropriate form are obtained as specializations on the generic Koszul differential module.

\begin{construction}[Generic Koszul Differential Module]\label{cons:genericKoszul}
Let $n \in \bbn$ and $A = \bbz [x_I \mid I \subset [n], \ I \neq \varnothing]$. Let $E = \bigoplus_{i=1}^n Ae_i$ and let $f_i \in \bigwedge^i E$ be the generic maps
$$f_i = \sum_{|I| = i} x_I e_I^*.$$
Next, let $I$ be the ideal generated the relations $f_i f_j = 0$ for $i, \ j$ both even. Then define $S := A / I$. Notice by construction $A / I \otimes K(f_1 , \dots , f_n)$ is a differential module with homology isomorphic to $\bbz$. 

Use the notation $K^{\textrm{gen}} := A/I \otimes K(f_1 , \dots , f_n)$.
\end{construction}

\begin{remark}
The relations induced by imposing the condition $f_i \cdot f_j = 0$ for $i$ and $j$ both even are in general quite complicated. The first case for which we obtain nontrivial equations is when $n = 4$ in Construction \ref{cons:genericKoszul}, in which case we are imposing the relation $f_2^{(2)} = 0$. Choosing bases, notice that $f_2$ is represented as a generic $4 \times 4$ skew symmetric matrix and the condition $f_2^{(2)} = 0$ means we are taking the quotient by the $4 \times 4$ pfaffian of this matrix representation.

Notice that in general, the ideal $I$ appearing in Construction \ref{cons:genericKoszul} is always generated by quadratic equations in the $f_i$.
\end{remark}

\begin{example}
Assume $n = 6$ in the notation of Construction \ref{cons:genericKoszul}. Then the relations imposed come from setting 
$$f_2^{(2)} = 0, \quad \textrm{and} \quad f_2 \cdot f_4 =0.$$
The relations $f_{2}^{(2)} = 0$ are precisely the equations of the $4 \times 4$ pfaffians of the $6\times 6$ matrix representation of the map $f_2$. The additional relation $f_2 \cdot f_4 = 0$ contributes, after choosing bases, the single quadratic equation
$$\sum_{\substack{I \subset [6], \\ |I| = 2}} \textrm{sgn} (I \subset [6]) x_I x_{[6] \backslash I} = 0.$$
\end{example}

It may be tempting to believe that all anchored free flags with complete intersection homology arise as specializations of Construction \ref{cons:genericKoszul}. We will see that this is true if the ring has characteristic $2$, but the following example shows that this is not the case in general.

\begin{example}
Let $E = \bigoplus_{i=1}^4 Re_i$, where $R = k[x_1 , \dots , x_4]$ and $k$ is a field of characteristic $\neq 2$. Let $F$ be the free flag defined by the following data:
$$A_{1,0} = A_{2,1} = A_{3,2} = A_{4,3} = x_1^3 e_1^* + x_2^3 e_2^* + x_3^3 e_3^* + x_4^3 e_4^*,$$
$$A_{2,0} = -A_{3,1} = A_{4,2} = x_1 x_2 e_{12}^* + x_2^2 e_{34}^*,$$
$$A_{3,0} = -2x_1 e_{134}^*, \quad A_{4,1} = A_{4,0} = 0.$$

Under a choice of basis, we can write $F = R^1\oplus F^4\oplus R^6\oplus R^4\oplus R^1$ and express the differential $d^F$ as a block matrix. 

\[
\begin{blockarray}{cccccc}
 & R^1 & R^4 & R^6 & R^4 & R^1\\
 \begin{block}{c(ccccc)}
 R^1 & 0 & A_{1,0} & A_{2,0} & A_{3,0} & A_{4,0}\\
 R^4 & 0 & 0 & A_{2,1} & A_{3,1} & A_{4,1}\\
 R^6 & 0 & 0 & 0 & A_{3,2} & A_{4,2}\\
 R^4 & 0 & 0 & 0 & 0 & A_{4,3}\\
 R^1 & 0 & 0 & 0 & 0 & 0 \\
 \end{block}
\end{blockarray}
\]

The blocks $A_{i,i-1}$ on the first off-diagonal are the matrices appearing in the Koszul complex on $(x_1^3, x_2^3, x_3^3, x_4^3)$. The $A_{i,i-2}$ maps are given by the following:

\[
A_{2,0} = \begin{pmatrix}
x_1x_2 & 0 & 0 & 0 & 0 & x_2^2
\end{pmatrix}\
\quad
A_{3,1} = \begin{pmatrix}
0 & 0 & -x_2^2 & 0\\
0 & 0 & 0 & -x_2^2\\
-x_1x_2 & 0 & 0 & 0\\
0 & -x_1x_2 & 0 & 0
\end{pmatrix}
\quad
A_{4,2} = \begin{pmatrix}
x_2^2\\ 0\\ 0\\ 0\\ 0\\ x_1x_2
\end{pmatrix}
\]

and we have $A_{3,0} = \begin{pmatrix}
0 & 0 & 2x_1 & 0
\end{pmatrix}$.

One can check that $d^F$ squares to 0 so this data determines a well-defined free flag. However, $A_{3,0}\neq A_{4,1}$ since $A_{3,0}$ is given by multiplication by a nonzero element and $A_{4,1}$ is not. This means that the free flag induced by the above data is not of the form $K(f_1,f_2,f_3,f_4)$ for any choice of $f_i$.
\end{example}

The above example hinges on the observation that if we let $A_{1,0} = \begin{pmatrix} x_1^3 & x_2^3 & x_3^3 & x_4^3 \end{pmatrix}$ and $A_{4,3} = \begin{pmatrix} -x_4^3 & x_3^3 & -x_2^3 & x_1^3\end{pmatrix}^T$ be the first and last matrices in the Koszul complex on $(x_1^3, x_2^3, x_3^3, x_4^3)$, then $A_{2,0}A_{4,2} + A_{3,0}A_{4,3} = 0$. The condition for $F$ to be a differential module requires that $A_{1,0}A_{4,1} + A_{2,0}A_{4,2} + A_{3,0}A_{4,3} = 0$, which in this case forces $A_{4,1}$ to be 0 modulo $(x_1^3, x_2^3, x_3^3, x_4^3)$. This ``forcing" occurred due to cancellation in $A_{2,0}A_{4,2} + A_{3,0}A_{4,3}$. Our next result shows that absent this cancellation, it is indeed the case that anchored free flags with complete intersection homolopgy as in Construction \ref{cons:genericKoszul} are Koszul differential modules.

This intuition is clarified in the following definition which presents a technical condition that will be needed as a hypothesis for the proof of Theorem \ref{thm:nonCHar2}. 

\begin{definition}
Let $R$ be a ring and $I \subset R$ an ideal. The ideal $I$ is \emph{completely Tor-independent} with respect to a family of ideals $J_1 , \dots , J_\ell$ if for all $1 \leq i_1 < \cdots < i_k \leq \ell$, one has
$$\tor_{>0}^R \Big( \frac{R}{I} , \frac{R}{J_{i_1} + \cdots + J_{i_k}} \Big) = 0.$$
\end{definition}

Conceptually, Tor-independent objects do not have any nontrivial homological interaction; at the level of resolutions, this means that tensoring a resolution of a module $M$ with a Tor-independent quotient ring preserves exactness.

The next theorem gives a partial characterization of anchored free flags, and is the main result of this section. The abridged version of this result states that if the homology $H(D)$ is sufficiently Tor-independent with respect to a subfamily of the ideals $\im (f_i \colon \bigwedge^i E \to R)$, then differential modules with complete intersection homology must arise as in Proposition \ref{prop:genKoszul}.

\begin{theorem}\label{thm:nonCHar2}
Let $D$ be a differential $R$-module with $H(D)$ a complete intersection. Let $F\to D$ be an anchored free flag resolution and assume that $\im (f_1 \colon E \to R)$ is completely Tor-independent with respect to the set
$$\Big\{ \im (f_i \colon \bigwedge^i E \to R) \mid i \ \textrm{is even and} \ i \leq \rank(E)/2 \Big\}.$$
Then $F$ is isomorphic to the differential module of Proposition \ref{prop:genKoszul}, where $T = \bigwedge^\bullet E$.
\end{theorem}

\begin{proof}
Proceed by induction on $i-j-1$. When $i-j-1 = 0$, this is the statement of Theorem \ref{thm:danMikeThm}. Assume now that $i-j-1 \geq 1$. Inducting also on $j$, one may assume that for all $k>0$, there is the equality $A_{i-k,j+1-k} = (-1)^{(i-k)(j-k)} f_{i-j-1}$ (the base case is for $k=j+1$, which holds by assumption). Using this, one computes:
\begingroup\allowdisplaybreaks
\begin{align*}
    0 &=\sum_{k=j+1}^{i-1} A_{k,j} A_{i,k}  \\
    &= 2 \cdot  \sum_{\substack{i-k, \ j-k \ \textrm{even} \\
    j<k \leq \lfloor (i+j)/2 \rfloor}} f_{k-j}\cdot f_{i-k} + f_1 \big( (-1)^{(i+1)j+i+j} f_{i-j-1} - A_{i,j+1} \big).
\end{align*}
\endgroup
If $i-j > \rank E$, then the above equality reduces to $f_1 \big( (-1)^{(i+1)j+i+j} f_{i-j-1} - A_{i,j+1} \big) = 0$, and exactness of multiplication by $f_1$ implies that $(-1)^{(i+1)j+i+j} f_{i-j-1} - A_{i,j+1} \in \im f_1$. If $i-j \leq \rank E$, then notice that for each $j < k \leq \lfloor (i+j)/2 \rfloor$, one has $i-k$ or $k-j \leq \rank(E)/2$. Define the ideal
$$\mfa := \big( \im (f_\ell \colon \bigwedge^{\ell} E \to R ) \mid \ell \ \textrm{is even and} \ \ell < i-j-1 \big) \subset R.$$
Let $I$ be any indexing set with $|I|=i-j-2$ and let $e_I^* \in \bigwedge^{i-j-2} E^*$ denote the basis element dual to $e_I \in \bigwedge^{i-j-2} E$. Multiplying the above equality on the right by $e_I^*$, one obtains
\begin{equation}\label{eqn:0Moduloa}
f_1 \cdot (b \cdot e_I^*) \in \mfa,
\end{equation}
where $b := (-1)^{(i+1)j+i+j} f_{i-j-1} - A_{i,j+1}$. By the Tor-independence assumption, notice that if $K_\bullet$ denotes the Koszul complex induced by $f_1$, then the complex $K_\bullet \otimes_R R/\mfa$ must remain exact. Combining this with \ref{eqn:0Moduloa} implies that $b \cdot e_I^*$ is a cycle in $K_\bullet \otimes_R R/\mfa$, and hence by exactness there exists some $a \in \bigwedge^2 E$ such that 
$$b\cdot e_I^* - f_1 \cdot a \in \mfa E.$$
Multiplying the above on the right by $e_i^*$ for $i=1 , \dots , n$, it follows that for all indexing sets $J$ of size $i-j-1$, there exist elements $a_\ell^J \in \bigwedge^{\ell} E$ such that
\begin{equation}\label{eqn:boundaryElements}
    b \cdot e_J^* = f_1 \cdot a_1^J + \sum_{\substack{\ell \ \textrm{even}, \\
\ell < i-j-1}} f_\ell \cdot a_\ell^J.
\end{equation}
Recall that the identity map $\bigwedge^{i-j-1} E \to \bigwedge^{i-j-1} E$ is equivalently represented as right multiplication by the trace element $\sum_{|J| = i-j-1} e_J^* \otimes e_J$, whence:
\begingroup\allowdisplaybreaks
\begin{align*}
    b &= b \cdot \Big( \sum_{|J| = i-j-1} e_J^* \otimes e_J \Big) \\
    &= \sum_{|J| = i-j-1} b e_J^* \otimes e_J \\
    &= \sum_{|J| = i-j-1} (f_1 \cdot a_1^J)e_J + \sum_{|J| = i-j-1} \sum_{\substack{\ell \ \textrm{even}, \\
\ell < i-j-1}} (f_\ell \cdot a_\ell^J) e_J \quad \textrm{(by \ref{eqn:boundaryElements})} \\
&= f_1 \cdot \Big( \sum_{|J| = i-j-1}  a_1^J \cdot e_J \Big) +  \sum_{\substack{\ell \ \textrm{even}, \\
\ell < i-j-1}} f_\ell \cdot \Big( \sum_{|J| = i-j-1}  a_\ell^J\cdot e_J \Big) \\
&\in \im \Big( f_1 \colon \bigwedge^{i-j} E \to \bigwedge^{i-j-1} E \Big) + \sum_{\substack{\ell \ \textrm{even}, \\
\ell < i-j-1}} \im \Big( f_\ell \colon \bigwedge^{i-j-1+\ell} E \to \bigwedge^{i-j-1} E \Big).
\end{align*}
\endgroup
The above computation thus shows that one may perform column operations on the matrix representation of the square-zero endomorphism of $D$ to ensure that  $A_{i,j+1} = (-1)^{(i+1)j+i+j} f_{i-j-1}$. This completes the proof.
\end{proof}

Although the hypotheses of Theorem \ref{thm:nonCHar2} are stated in a decent level of generality, the following corollary makes explicit a list of common cases for which the hypotheses of Theorem \ref{thm:nonCHar2} are satisfied:

\begin{cor}
The assumptions of Theorem \ref{thm:nonCHar2} are satisfied in the following cases:
\begin{enumerate}
    \item The free module $E$ satisfies $\rank E \leq 3$.
    \item $(R, \m)$ is a Noetherian local ring and
    $$\im \Big( f_1 \colon E \to R \Big) + \sum_{\substack{\ell \leq \rank(E)/2 \\
    \ell \ \textrm{even}}} \im \Big( f_\ell \colon \bigwedge^\ell E \to R \Big)$$
    is generated by a regular sequence contained in $\m$.
    \item $R$ is a graded ring and
    $$\im \Big( f_1 \colon E \to R \Big) + \sum_{\substack{\ell \leq \rank(E)/2 \\
    \ell \ \textrm{even}}} \im \Big( f_\ell \colon \bigwedge^\ell E \to R \Big)$$
    is generated by a homogeneous regular sequence of positive degree.
    \item $R = k[x_1 , \dots , x_n]$ and the image of each $f_i \colon \bigwedge^i E \to R$ for $i=1$ and $i \leq \rank(E)/2$ even lie in polynomial rings in disjoint variables. 
\end{enumerate}
In particular, if either:
\begin{enumerate}
    \item $(R,\m)$ is a Noetherian local ring and the first row of the matrix representation of the square-zero endomorphism of $D$ generates a complete intersection contained in $\m$, or
    \item $R$ is a graded ring and the first row of the matrix representation of the square-zero endomorphism of $D$ generates a homogeneous complete intersection of positive degree,
\end{enumerate}
then the assumptions of Theorem \ref{thm:nonCHar2} are satisfied.
\end{cor}

Interestingly, if we assume that $R$ is has characteristic $2$, then the statement of Theorem \ref{thm:nonCHar2} can be generalized significantly.

\begin{theorem}
Assume $R$ is a ring of characteristic $2$. Let $F$ be an anchored free flag with $H(F)$ a complete intersection. Then $F$ isomorphic to the differential module of Proposition \ref{prop:genKoszul} for some choice of $f_i \in \bigwedge^i E^*$, where $T = \bigwedge^\bullet E$.
\end{theorem}

\begin{proof}
Do the computation of the previous proof, but notice that all other extraneous terms cancel by the characteristic assumption.
\end{proof}


\section{DG-module Structures on Free Flags}\label{sec:DGAlg}

In this section, we study differential modules that can be given the structure of a DG-module over some DG-algebra. Our main motivation for considering this question is based on the philosophy that the homological properties of a differential module are tightly linked to those of the homology, as suggested by Theorem \ref{thm:danMikeThm}. One natural direction related to this question is the extent to which additional structure on the minimal free resolution of the homology can be ``lifted" to the differential module. Our results here indicate that there are some very restrictive obstructions to lifting algebra structures to the level of differential modules.

It is evident that if a free resolution $F_\bullet$ of the homology $H(D)$ of an anchored free flag $D$ admits the structure of an associative DG-algebra structure and $D \cong \fold (F_\bullet)$, then the algebra structure on $F_\bullet$ can be transferred to a DG-module structure on $D$. We prove even further that if the homology $H(D)$ is a complete intersection, then this becomes an equivalence; more precisely: an anchored free flag $D$ with $H(D)$ a complete intersection admits the structure of a DG-module over $K_\bullet$ \emph{if and only if} $D \cong \fold (K_\bullet)$, where $K_\bullet$ denotes the Koszul complex resolving $H(D)$. 

We conclude the section with questions about DG-module structures on more general free flag resolutions. In particular, we know of no example of a free flag admitting a DG-module structure over the minimal free resolution of its homology that is \emph{not} isomorphic to the fold of some complex, and are very interested in any such example. 


\begin{definition}\label{def:dga}
A \emph{differential graded algebra} $(F,d)$ (or \emph{DG-algebra}) over a commutative Noetherian ring $R$ is a complex of finitely generated free $R$-modules with differential $d$ and with a unitary, associative multiplication $F \otimes_R F \to F$ satisfying
\begin{enumerate}[(a)]
    \item $F_i F_j \subseteq F_{i+j}$,
    \item $d_{i+j} (f_i f_j) = d_i (f_i) f_j + (-1)^i f_i d_j (f_j)$,
    \item $f_i f_j = (-1)^{ij} f_j f_i$, and
    \item $f_i^2 = 0$ if $i$ is odd,
\end{enumerate}
where $f_k \in F_k$.
\end{definition}

There does not exist a tensor product between arbitrary differential modules (or even free flags) that directly generalizes the tensor product of complexes, but, it is possible to construct such a product between a \emph{complex} and a differential module.

\begin{definition}
Let $F_\bullet$ be a complex and $D$ a differential module. Then the \emph{box product} $F_\bullet \boxtimes_R D$ is defined to be the differential module with underlying module $\bigoplus_{i\in \bbz} F_i \otimes_R D$ and differential
$$d^{F \boxtimes D} (f_i \otimes d) := d^F (f_i) \otimes d + (-1)^i f_i \otimes d^D (d).$$
\end{definition}

\begin{definition}\label{def:DGmod}
Let $D$ be a differential module and let $F_\bullet$ be a minimal free resolution of $H(D)$ admitting the structure of a DG-algebra. Then $D$ is a DG-module over $F_\bullet$ if there exists a morphism of differential modules
$$p \colon F_\bullet \boxtimes_R D \to D.$$
In such a case, the notation $f_i \cdot_D d := p(f_i \otimes d)$ will be used. 
\end{definition}

\begin{remark}
Sometimes the simpler notation $\cdot$ will be used over $\cdot_D$ when it is clear which product is being considered. It is important to note that there is almost no hope for an appropriate generalization of a DG-algebra even for general free flag resolutions. This is for at least two reasons: firstly, as already mentioned, there is no natural candidate for the tensor product of two differential modules, so one cannot employ a definition similar to Definition \ref{def:DGmod}. Secondly, the ``degree" of an element is not well-defined if it is induced by the flag grading of arbitrary free flag $D = \bigoplus_{i \in \bbz} F_i$, since any given $f_i \in F_i$ is contained in $D^j$ for all $j \geq i$.
\end{remark}



\begin{obs}\label{obs:foldOfDG}
If $F_\bullet$ is a complex admitting the structure of a DG-algebra, then $\fold (F_\bullet)$ is a DG-module over $F_\bullet$.
\end{obs}

\begin{proof}
Just define the product on $\fold (F_\bullet)$ via the product on $F_\bullet$.
\end{proof}

\begin{prop}\label{prop:indAlg}
Let $\phi \colon D \to D'$ be an isomorphism of differential modules and $F_\bullet$ a DG-algebra minimal free resolution of $H(D)$. If $D'$ is a DG-module over $F_\bullet$, then $D$ is a DG-module over $F_\bullet$ with the induced product:
$$f_i \cdot_D d := \phi^{-1} \big( f_i \cdot_{D'} \phi(d) \big).$$
Moreover, $\phi$ becomes a morphism of DG-modules with this product.
\end{prop}

\begin{proof}
The induced product is defined by making the following diagram commute:
\[\begin{tikzcd}
	{F_\bullet \boxtimes D} & {F_\bullet \boxtimes D'} \\
	D & {D'}
	\arrow[from=1-1, to=2-1]
	\arrow["{1 \boxtimes \phi}", from=1-1, to=1-2]
	\arrow["{\cdot_{D'}}", from=1-2, to=2-2]
	\arrow["{\phi^{-1}}", from=2-2, to=2-1]
\end{tikzcd}\]
\end{proof}

\begin{example}\label{ex:dgmod}
Let $R = k[x_1,x_2]$ and $E$ be a rank $2$ free module on the basis $e_1 , e_2$. Let $D = \bigwedge^\bullet E$ be the free flag with differential
$$\begin{pmatrix} 0 & x_1 & x_2 & x_1^2+x_2^2 \\
0 & 0 & 0 & -x_2 \\
0 & 0 & 0 & x_1 \\
0 & 0 & 0 & 0  \\
\end{pmatrix}.$$
Then $D$ admits the structure of a DG-module over the Koszul complex $K_\bullet$ with the following product (any product not listed is understood to be $0$):
$$1 \cdot_D d = d \ \textrm{for all} \ d \in D, \quad e_1 \cdot_D 1 = e_1, \ e_2 \cdot_D 1 = e_2, $$
$$e_{12} \cdot_D 1 = e_{12} - x_1 e_1 - x_2 e_2,$$
$$e_1 \cdot_D e_2 = e_{12} - x_1 e_1 - x_2 e_2,$$
$$e_1 \cdot_D e_{12} = x_2 e_{12} - x_1 x_2 e_1 - x_2^2 e_2, \quad \textrm{and}$$
$$e_2 \cdot_D e_{12} = -x_1 e_{12} + x_1^2 e_1 + x_1 x_2 e_2.$$

\end{example}

Recall that by Theorem \ref{cor:CIequiv}, $D$ in Example \ref{ex:dgmod} isomorphic to the fold of the Koszul complex on $(x_1, x_2)$. In fact, the product given above is induced by this isomorphism. This construction works in general, that is if we have an isomorphism to a $DG$-algebra, we can obtain a $DG$-module structure in the same way. This leads to a string of immediate corollaries to Proposition \ref{prop:indAlg}.

\begin{cor}\label{cor:indAlgFromFold}
Let $D$ be an anchored free flag and assume that $D \cong \fold (F_\bullet)$ where $F_\bullet\to H(D)$ is a minimal free resolution. If $F_\bullet$ admits the structure of a DG-algebra, then $D$ admits the structure of a DG-module over $F_\bullet$.
\end{cor}

\begin{proof}
This follows from Proposition \ref{prop:indAlg} and Observation \ref{obs:foldOfDG}.
\end{proof}

\begin{cor}\label{cor:imageContainCor}
Let $D$ be an anchored free flag and assume that the minimal free resolution $F_\bullet$ of $H(D)$ admits the structure of a DG-algebra. If the matrices $A_{i,0}$ satisfy $\im A_{i,0} \subset \im d_1$ for each $i \geq 2$, then $D$ is a DG-module over $F_\bullet$.
\end{cor}

\begin{proof}
If $\im A_{i,0} \subset \im d_F$, then $D \cong \fold (F_\bullet)$ by Lemma \ref{lem:MayaLem}, so employ Corollary \ref{cor:indAlgFromFold}.
\end{proof}

\begin{cor}
Let $D$ be a degree $0$ anchored free flag with $H(D) \cong k$. Then $D$ admits the structure of a DG-module over the minimal free resolution of $k$.
\end{cor}

\begin{proof}
The asssumption that $D$ has degree $0$ implies that each matrix $A_{i,0}$ has entries in $\m$, so employ Corollary \ref{cor:imageContainCor}.
\end{proof}

The above string of corollaries are all proved by reducing to the case that the differential module being considered may be realized as the folding of a DG-algebra resolution. The following theorem shows that this assumption is not only sufficient, but also \emph{necessary} for free flags with complete intersection homology.  

\begin{theorem}\label{thm:noDGmod}
Let $D$ be an anchored free flag with $H(D)$ a complete intersection. Let $K_\bullet$ denote the Kosul complex resolving $H(D)$. Then $D$ is a DG-module over $K_\bullet$ if and only if $D \cong \fold (K_\bullet)$.
\end{theorem}

\begin{proof}
$\impliedby$: This is clear.

$\implies$: Choose $i := \min \{ i > 1 \mid f_i \neq 0 \}$; if no such $i$ exists, then $D \cong \fold (K_\bullet)$ by Lemma \ref{lem:MayaLem} and there is nothing to prove, so assume $i$ exists. Otherwise, Lemma \ref{lem:MayaLem} implies that we may assume $A_{k,j} = 0$ for all $k-j < i$. In particular, for all indexing sets $I$ of size $i$, one has
$$d (e_I ) = f_1 (e_I) + f_i (e_I).$$
Assume that $D$ has a DG-module structure over $K_\bullet$. By assumption $H(D) \cong R/ \mfa$ for some ideal $\mfa$ that is generated by a regular sequence, and $f_1 \otimes R/\mfa = 0$. One can choose the algebra structure such that $e_I \cdot e_J = e_I \w e_J + t_{i-1}$ for all $|I| + |J| \leq i$, where $t_{i-1}$ is some element of $\bigoplus_{j \leq i-1} \bigwedge^j E$. This is because exactness of multiplication by $f_1$ forces $e_I \w e_J - p_{|I|+|J|}(e_I \cdot_D e_J)$ to be a boundary in the Koszul complex induced by $f_1$, where $p_{|I|+|J|} \colon D \to \bigwedge^{|I|+|J|} E$ denotes the projection onto the corresponding direct summand. Let $e_I$ be any basis vector such that $f_i (e_I) \otimes R/\mfa \neq 0$ (such an element must exist by selection of $i$), where $|I|= i$. Let $\ell$ be the first element of $I$ and notice that:
\begingroup\allowdisplaybreaks
\begin{align*}
    f_1 (e_I) + f_i (e_I) &= d(e_I) \\
    &= d(e_\ell \cdot_D e_{I \backslash \ell} + t_{i-1} ) \\
    &= d(e_\ell) \cdot_D e_{I \backslash \ell} - e_\ell \cdot_D d(e_{I \backslash \ell}) + d(t_{i-1}).
\end{align*}
\endgroup
Tensoring the above relation with $R / \mfa$, it follows that $f_i (e_I) \in \mfa$, which is a contradiction to the assumptions. It follows that no DG-module structure can exist.
\end{proof}

In view of Theorem \ref{thm:noDGmod}, it follows that DG-module structures over the minimal free resolution of the homology are actually quite rare. Indeed, after running many examples it seems that the only time such a DG-module structure exists is if the free flag arises as the fold of the minimal free resolution, indicating that DG-module structures can be used to distinguish free flags that are in the isomorphism class of a complex. It is an interesting question as to whether there exists a family of structures, similar to a DG-module structure, that can be used to detect the isomorphism class of any anchored free flag. We conclude with the following (likely easier) question:

\begin{question}\label{question:nontrivialDG}
Does there exist an anchored free flag $D$ admitting the structure of a module over the minimal free resolution of its homology $F_\bullet$ that is \emph{not} isomorphic to $\fold (F_\bullet)$? 
\end{question}

\bibliographystyle{amsalpha}
\bibliography{biblio}

\providecommand{\bysame}{\leavevmode\hbox to3em{\hrulefill}\thinspace}
\providecommand{\MR}{\relax\ifhmode\unskip\space\fi MR }
\providecommand{\MRhref}[2]{%
  \href{http://www.ams.org/mathscinet-getitem?mr=#1}{#2}
}
\providecommand{\href}[2]{#2}
\begin{thebibliography}{XYY15}

\bibitem[ABI07]{avramov2007class}
Luchezar~L. Avramov, Ragnar-Olaf Buchweitz, and Srikanth Iyengar, \emph{Class
  and rank of differential modules}, Invent. Math. \textbf{169} (2007), no.~1,
  1--35. \MR{2308849}

\bibitem[BD10]{boocher2010rank}
Adam Boocher and Justin~W DeVries, \emph{On the rank of multigraded
  differential modules}, arXiv preprint arXiv:1011.2167 (2010).

\bibitem[BE77]{buchsbaum1977algebra}
David~A Buchsbaum and David Eisenbud, \emph{Algebra structures for finite free
  resolutions, and some structure theorems for ideals of codimension 3},
  American Journal of Mathematics \textbf{99} (1977), no.~3, 447--485.

\bibitem[BE21a]{brown2021minimal}
Michael~K Brown and Daniel Erman, \emph{Minimal free resolutions of
  differential modules}, arXiv preprint arXiv:2103.08471 (2021).

\bibitem[BE21b]{brown2021tate}
Michael~K. Brown and Daniel Erman, \emph{Tate resolutions on toric varieties},
  arXiv preprint arXiv:2108.03345 (2021).

\bibitem[Car86]{10.1007/BFb0072815}
Gunnar Carlsson, \emph{Free $\bbz/2$k-actions and a problem in commutative
  algebra}, Transformation Groups Pozna{\'{n}} 1985 (Berlin, Heidelberg)
  (Stefan Jackowski and Krzysztof Pawa{\l}owski, eds.), Springer Berlin
  Heidelberg, 1986, pp.~79--83.

\bibitem[CE16]{cartan2016homological}
Henry Cartan and Samuel Eilenberg, \emph{Homological algebra (pms-19), volume
  19}, Princeton university press, 2016.

\bibitem[Har79]{HARTSHORNE1979117}
Robin Hartshorne, \emph{Algebraic vector bundles on projective spaces: A
  problem list}, Topology \textbf{18} (1979), no.~2, 117--128.

\bibitem[IW18]{iyengar2018examples}
Srikanth~B Iyengar and Mark~E Walker, \emph{Examples of finite free complexes
  of small rank and small homology}, Acta Mathematica \textbf{221} (2018),
  no.~1, 143--158.

\bibitem[Rou06]{rouquier2006representation}
Rapha{\"e}l Rouquier, \emph{Representation dimension of exterior algebras},
  Inventiones mathematicae \textbf{165} (2006), no.~2, 357--367.

\bibitem[RZ17]{ringel2017representations}
Claus~Michael Ringel and Pu~Zhang, \emph{Representations of quivers over the
  algebra of dual numbers}, Journal of Algebra \textbf{475} (2017), 327--360.

\bibitem[Sta17]{stai2017differential}
Torkil Stai, \emph{Differential modules over quadratic monomial algebras},
  Algebras and Representation Theory \textbf{20} (2017), no.~5, 1239--1247.

\bibitem[{\c{S}}{\"U}19]{csenturk2019carlsson}
Berrin {\c{S}}ent{\"u}rk and {\"O}zg{\"u}n {\"U}nl{\"u}, \emph{Carlsson's rank
  conjecture and a conjecture on square-zero upper triangular matrices},
  Journal of Pure and Applied Algebra \textbf{223} (2019), no.~6, 2562--2584.

\bibitem[Wei15]{wei2015gorenstein}
Jiaqun Wei, \emph{Gorenstein homological theory for differential modules},
  Proceedings of the Royal Society of Edinburgh Section A: Mathematics
  \textbf{145} (2015), no.~3, 639--655.

\bibitem[XYY15]{xu2015gorenstein}
Huabo Xu, Shilin Yang, and Hailou Yao, \emph{Gorenstein theory for n-th
  differential modules}, Periodica Mathematica Hungarica \textbf{71} (2015),
  no.~1, 112--124.

\end{thebibliography}
\addcontentsline{toc}{section}{Bibliography}

\end{document}